\documentclass[10pt]{extarticle}
\usepackage{graphicx}
\usepackage{amssymb}
\usepackage{amsfonts,amsmath,amsxtra,amsthm,amssymb,latexsym}
\usepackage[utf8]{inputenc}
\usepackage{color,xcolor}
\usepackage[utf8]{inputenc}
\usepackage[T2A]{fontenc}
\usepackage[russian,english]{babel}
\DeclareMathOperator{\card}{card}

\newtheorem{theorem}{Theorem}[section]

\newtheorem{lemma}[theorem]{Lemma}
\newtheorem{corollary}[theorem]{Corollary}
\newtheorem{remark}[theorem]{Remark}

\begin{document}

\begin{center}
{\bf Error Analysis of Truncation Legendre Method \\for Solving Numerical Differentiation}
\end{center}

\vspace*{5mm}
\centerline{\textsc{M. Kyselov$\!\!{}^{\dag}$}}

\vspace*{5mm}
\centerline{$\!\!{}^{\dag}\!\!$ Institute of Mathematics, National Academy of Sciences of Ukraine, Kyiv}

\begin{abstract}
We study the problem of numerical differentiation of functions from weighted Wiener classes.
We construct and analyze a truncation Legendre method to recover arbitrary order derivatives.
The main focus is on obtaining its error estimates in integral and uniform metrics.

\vspace*{3mm}


\vspace*{3mm}


\end{abstract}

\section{Introduction. Description of the problem}


The problem of numerical differentiation poses a central challenge in computational mathematics, 
with widespread applications across scientific research and engineering disciplines. 
This challenge gains particular prominence in situations where analytical computation of function derivatives 
presents difficulties or proves impossible, especially when working with approximate or noisy functional data. 
It is noteworthy that numerical differentiation, despite its extensive historical development, continues to attract 
significant attention within the mathematical research community (see, for example,
\cite{Dolgopolova&Ivanov_USSR_Comput_Math_Math_Phys_1966_Eng}, \cite{Ramm_1968_No11}, \cite{VasinVV_1969_V7_N2},
\cite{Cul71}, \cite{And84}, \cite{EgorKond_1989}, \cite{Groetsch_1992_V74_N2},  \cite{Qu96},
\cite{Hanke&Scherzer_2001_V108_N6}, \cite{RammSmir_2001}, \cite{Ahn&Choi&Ramm_2006},
\cite{ Wang_Hon_Ch_2006},  \cite{Nakamura&Wang&Wang_2008},\cite{Zhao&Meng&Zhao&You&Xie_2016}, \cite{SSS_CMAM}).

The fundamental challenge associated with numerical differentiation lies in its inherent instability, 
rendering it ill-posed according to Hadamard's definition.
Specifically, small perturbations in input data can result in substantial errors when evaluating the deviation 
of approximate derivatives from exact ones.

The present investigation focuses on the truncated Legendre method for addressing numerical differentiation problems.
It is important to note that this method was initially examined in the context of numerical differentiation
in \cite{Zhao_2010} and \cite{Lu&Naum&Per}.
This approach has demonstrated several beneficial characteristics, notably its straightforward algorithmic implementation 
and the elimination of requirements for solving equation systems.
Consequently, research of the truncated Legendre method has been extended to various classes of differentiable functions
(see, e.g., \cite{Sol_Stas_UMZ2022}, \cite{Sem_Sol_ETNA}, \cite{Sem_Sol_CMAM2024}).

The aim of our investigation is to derive error bounds for this method when applied to functions belonging 
to weighted Wiener classes for the reconstruction of their derivatives of arbitrary order.
In contrast to earlier investigations, which predominantly concentrated on first-order derivatives and 
their reconstruction in the metric of specific $C$ and $L_2$ spaces, we conduct a thorough analysis 
across an extensive range of parameters characterizing the function class under investigation, derivative order, 
and metrics of both input and output spaces.

In developing our proposed approach, particular emphasis is placed on identifying the optimal value of 
the truncation parameter, which minimizes the combined error arising from series truncation and input data perturbations. 
We establish explicit relationships between this parameter and the input data error level, function smoothness, 
and the order of the derivative.

Our findings possess both theoretical importance, which is crucial for comprehending the fundamental nature 
of numerical differentiation problems, and practical significance in creating stable computational algorithms 
for diverse applied problems.

We note that our investigation generalizes existing results to encompass derivatives of arbitrary order and 
various output metrics (integral and uniform), enabling a more comprehensive characterization of the capabilities 
and effective application conditions of the truncated expansion method for numerical differentiation.

To establish our theoretical framework, we introduce the essential mathematical concepts.
We denote by $L_{q}$, $2\le q\le \infty$, the space of real-valued functions $f(t)$
that are integrable to the $q$-th power on the interval $[-1, 1]$,
with norms specified as:
$$
\|f\|_{q}^q := \int_{-1}^1 |f(t)|^q \, d t < \infty ,
\quad 2 \le q < \infty,
$$
$$
\|f\|_{\infty}:=
\mbox{ess}  \sup\limits_{\! \! \! \! \! \! \! \! \! \! \! \! \! t \in [-1, 1]} |f(t)| < \infty ,
\quad q = \infty .
$$

We employ the system of orthonormal Legendre polynomials $\{\varphi_k(t)\}_{k=0}^\infty$ given by:
$$
\varphi_k(t) = \sqrt{k + 1/2}(2^k k!)^{-1} \frac{d^k}{dt^k}[(t^2 - 1)^k], \quad k = 0,1,2,\dots .
$$
In the Hilbert space $L_{2}$, the norm may be represented using these Legendre polynomials as:
$$
\|f\|_2^2 = \sum_{k=0}^{\infty}|\langle f, \varphi_k \rangle|^2 ,
$$
where the Fourier-Legendre coefficients are defined by:
$$
\langle f, \varphi_k\rangle=\int_{-1}^{1} f (t)\, \varphi_k(t)\, dt,
\quad k=0,1,2,\ldots .
$$

Furthermore, we denote $C$ as the space of continuous real-valued functions on $[-1,1]$
and $\ell_p$, $1\leq p\leq\infty$, as the space of numerical sequences
$\overline{x}=\{x_{k}\}_{k\in\mathbb{N}_0}$ that satisfy:
$$
\|\overline{x}\|_{\ell_p}  := \left\{
\begin{array}{cl}
\bigg(\sum\limits_{k\in\mathbb{N}_0} |x_{k}|^p\bigg)^{\frac{1}{p}} < \infty ,
 \ & 1\leq p<\infty ,
\\\\
\sup\limits_{k\in\mathbb{N}_0}  |x_{k}| < \infty ,
  \ & p=\infty .
\end{array}
\right.
$$

For our theoretical analysis, we introduce the functional spaces:
$$
W_{s}^\mu =\{f\in L_{2}: \quad
\|f\|_{s,\mu}^s :=\sum_{k=0}^{\infty} ({\max\{1,k\}})^{s\mu}|\langle f,
\varphi_k\rangle|^s<\infty\},
$$
where $\mu>0$ and $1\le s<\infty$. We employ the same notation for both the space and its unit ball: $W_{s}^{\mu} = \{f\in W_{s}^{\mu}\!: \|f\|_{s,\mu} \leq 1\}$, which we term a function class. The specific interpretation, space or class, will be clear from the context.
These functional spaces $W_{s}^\mu$ are extensively studied in approximation theory and are recognized as weighted Wiener classes (as discussed in \cite{Kolom2023}).

\begin{remark} \label{Tom}
We introduce $L^r_{2}$, $r=1,2,\ldots$, as the space of functions $f(t)$, where $f, f',\ldots, f^{(r-1)}$
demonstrate absolute continuity on any interval $(\eta_1,\eta_2)$ with $\eta_1>-1$ and $\eta_2<1$, and satisfy:
$$\int_{-1}^{1} (1-t)^{r} (1+t)^{r} |f^{(r)}(t)|^2 \, d t < \infty .$$
Based on \cite{Tom}, the space $L^r_{2}$ can be endowed with the norm:
$$
\|f\|_{L^r_{2}} := \Big(\sum_{j=0}^r \int_{-1}^{1} (1-t)^{j} (1+t)^{j} |f^{(j)}(t)|^2 dt\Big)^{1/2}
$$
and the following equivalence is established:
1) $f\in L^r_{2}$ if and only if 2) $f\in W^r_{2}$.
\end{remark}

Any function $f(t)$ from $W_{s}^{\mu}$ admits the representation:
\begin{equation}\label{function}
	f(t) = \sum_{k=0}^{\infty}\langle f, \varphi_k \rangle \varphi_k(t),
\end{equation}
while its $r$-th order derivative with $r\in \mathbb{N}$ can be written as:
\begin{equation}\label{r_deriv}
	f^{(r)}(t) = \sum_{k=r}^{\infty}\langle f, \varphi_k \rangle \varphi_k^{(r)}(t).
\end{equation}

In our problem setup, we assume that the input error measurements are conducted using the $\ell_p$ metric.
More precisely, we will deal with a sequence of real numbers $\overline{f^\delta}= \{\langle
f^\delta, \varphi_{k} \rangle\}_{k\in\mathbb{N}_0}$ such that, for $\overline{\xi}=
\{\xi_{k}\}_{k\in\mathbb{N}_0}$, $\xi_{k}=\langle f-f^\delta, \varphi_{k}\rangle$, and for some $1\leq
p\leq \infty$, the following constraint is satisfied:
\begin{equation}\label{perturbation2}
\|\overline{\xi}\|_{\ell_p} \leq \delta , \quad 0<\delta <1 .
\end{equation}

\section{Truncation method. Error estimate in the uniform metrics}
For reconstructing the $r$-th derivative (\ref{r_deriv}) of functions $f\in W_s^\mu$, $r=1,2,\ldots$,
we implement the truncation method:
\begin{equation}\label{ModVer}
	\mathcal{D}_N^{(r)} f^\delta(t) = \sum_{k=r}^{N} \langle f^\delta, \varphi_k\rangle \varphi_k^{(r)}(t).
\end{equation}

For our calculations, we need the following formula for differentiation of Legendre polynomials:
\begin{equation} \label{dif_ort}
\varphi'_k(t) = 2\sqrt{k + 1/2}\mathop{{\sum}^*}\limits_{l=0}^{k-1} \sqrt{l + 1/2}\varphi_l(t), \quad k \in \mathbb{N} ,
\end{equation}
where in aggregate \quad $ \mathop{{\sum}^*}\limits_{l=0}^{k-1} \bar{\varphi}_{l}(t)$ the summation
is extended over only those terms for which $k+l$ is odd.

Let us write the error of the method (\ref{ModVer}) as
\begin{equation}\label{fullError}
	f^{(r)}(t)-\mathcal{D}_N^{(r)} f^\delta(t)= \left(f^{(r)}(t)-\mathcal{D}_N^{(r)}
	f(t)\right)+\left(\mathcal{D}_N^{(r)} f(t)-\mathcal{D}_N^{(r)} f^\delta(t)\right)
\end{equation}
with
$$
\mathcal{D}_N^{(r)} f(t) = \sum_{k=r}^{N} \langle f, \varphi_{k}\rangle \varphi^{(r)}_k(t) .
$$

The parameter $N$ should be chosen depending on $\delta$, $p$, $s$ and $\mu$ so as
to minimize the error estimate for the method (\ref{ModVer}).

We use the notation $A\asymp B$ when there are absolute constants $C_1, C_2>0$ such that $C_1 B \le A \le C_2 B$ is true.

An upper bound for the norm of the first difference on the right-hand side of (\ref{fullError}) is contained in the following statement.

\begin{lemma}\label{lemma_BoundErrHCC}
	Let $f\in W^\mu_{s}$, $1\leq s< \infty$, $\mu>2r-1/s+3/2$. Then it holds
	$$
	\|f^{(r)}-\mathcal{D}^{(r)}_N f\|_{C}\leq c\|f\|_{s,\mu} N^{-\mu+2r-1/s+3/2} .
	$$
\end{lemma}

\begin{proof}
Using (\ref{dif_ort}), we write down the representation
$$
f^{(r)}(t)-\mathcal{D}^{(r)}_N f(t) =
\sum_{k=N+1}^\infty \langle f, \varphi_{k} \rangle \varphi_k^{(r)} (t) =
2 \sum_{k=N+1}^\infty \sqrt{k+1/2}\, \langle f, \varphi_{k} \rangle \mathop{{\sum}^*}\limits_{l_1=0}^{k-1}  \sqrt{l_1+1/2} \ \varphi_{l_1}^{(r-1)}(t)= \ldots
$$
$$
	 = 2^r \sum_{k=N+1}^\infty \sqrt{k+1/2}\, \langle f, \varphi_{k} \rangle
	\mathop{{\sum}^*}\limits_{l_1=r-1}^{k-1} (l_1+1/2) \mathop{{\sum}^*}\limits_{l_2=r-2}^{l_1-1} (l_2+1/2)
$$
\begin{equation}  \label{first_dif}
    \ldots
	\mathop{{\sum}^*}\limits_{l_{r-1}=1}^{l_{r-2}-1} (l_{r-1}+1/2)
	\mathop{{\sum}^*}\limits_{l_{r}=0}^{l_{r-1}-1} \sqrt{l_r+1/2} \ \varphi_{l_r}(t)  .
\end{equation}

It is well-known (see for example \cite{Suetin}, p. 128) that
$$
\max_{-1\le t\le 1} |\varphi_k(t)| = \varphi_k(1) := \sqrt{k+1/2},\quad k=0,1,2,\ \ldots .
$$

Now we can estimate
\begin{equation}
\begin{split}
\|f^{(r)} - D_N^{(r)} f\|_C &\leq 2^r \sum_{k=N+1}^{\infty} \sqrt{k + 1/2}\,|\langle f, \varphi_k\rangle| \sum_{l_1=r-1}^{k-1} (l_1 + 1/2) \sum_{l_2=r-2}^{l_1-1} (l_2 + 1/2)  \\
&\quad \ldots \sum_{l_{r-2}=1}^{l_{r-3}-1} (l_{r-2} + 1/2) \sum_{l_{r-1}=1}^{l_{r-2}-1} (l_{r-1} + 1/2) \sum_{l_r=0}^{l_{r-1}-1} \sqrt{l_r + 1/2}\,\|\varphi_{l_r}\|_C \\
&\leq 2^r \sum_{k=N+1}^{\infty} \sqrt{k + 1/2}\,|\langle f, \varphi_k\rangle| \sum_{l_1=r-1}^{k-1} (l_1 + 1/2) \sum_{l_2=r-2}^{l_1-1} (l_2 + 1/2)  \\
&\quad \ldots \sum_{l_{r-2}=1}^{l_{r-3}-1} (l_{r-2} + 1/2) \sum_{l_{r-1}=1}^{l_{r-2}-1} (l_{r-1} + 1/2) \sum_{l_r=0}^{l_{r-1}-1} (l_r + 1/2) \\
& \leq c \sum_{k=N+1}^{\infty} \sqrt{k + 1/2} \ k^{2r}|\langle f, \varphi_k\rangle| = c \sum_{k=N+1}^{\infty} k^{2r+1/2}|\langle f, \varphi_k\rangle|.
\end{split}
\end{equation}

Let's start with the case $1< s< \infty$. Then using the Hölder inequality and the definition of $W^\mu_s$, we continue for $\mu>2r-1/s+3/2$
$$
\|f^{(r)}-\mathcal{D}^{(r)}_N f\|_C \leq c \sum_{k=N+1}^{\infty} k^\mu|\langle f, \varphi_k\rangle|k^{2r+1/2-\mu}
$$
$$
\leq c\left(\sum_{k=N+1}^{\infty} k^{s\mu}|\langle f, \varphi_k\rangle|^s\right)^{1/s} \left(\sum_{k=N+1}^{\infty} k^{(2r+1/2-\mu)\frac{s}{s-1}}\right)^{(s-1)/s}
$$
$$
\leq c\|f\|_{s,\mu} N^{-\mu+2r-1/s+3/2}.
$$
In the case of $s = 1$ the assertion of Lemma is proved similarly.
\end{proof}

\begin{lemma}\label{lemma_BoundPertHCC}
	Let the condition (\ref{perturbation2}) be satisfied for $1\le p\le \infty$. Then for arbitrary function $f\in L_2$ it holds
	$$
	\|\mathcal{D}^{(r)}_N f - \mathcal{D}^{(r)}_N f^\delta\|_C \leq c\delta N^{2r-1/p+3/2}.
	$$
\end{lemma}

\begin{proof}
Using (\ref{dif_ort}), we write down the representation
$$
\mathcal{D}^{(r)}_N f(t) - \mathcal{D}^{(r)}_N f^\delta(t)
= \sum_{k=r}^N  \langle f-f^\delta, \varphi_{k} \rangle \varphi_{k}^{(r)}(t)
= 2 \sum_{k=r}^N \sqrt{k+1/2}\, \langle f-f^\delta, \varphi_{k} \rangle
\mathop{{\sum}^*}\limits_{l_1=0}^{k-1} \varphi_{l_1}^{(r-1)}(t)
$$
$$
	  \cdots =2^r \sum_{k=r}^N \sqrt{k+1/2}\, \langle f-f^\delta, \varphi_{k} \rangle
	\mathop{{\sum}^*}\limits_{l_1=r-1}^{k-1} (l_1+1/2) \mathop{{\sum}^*}\limits_{l_2=r-2}^{l_1-1} (l_2+1/2)
$$
\begin{equation}  \label{second_dif}
    \ldots
	\mathop{{\sum}^*}\limits_{l_{r-1}=1}^{l_{r-2}-1} (l_{r-1}+1/2)
	\mathop{{\sum}^*}\limits_{l_r=0}^{l_{r-1}-1} \sqrt{l_r+1/2} \ \varphi_{l_r}(t) .
\end{equation}
Let $1< p<\infty$ first. Then, using the H\"{o}lder inequality, we find
$$
\|\mathcal{D}^{(r)}_N f - \mathcal{D}^{(r)}_N f^\delta\|_{C}
\leq 2^r \sum_{k=r}^N \sqrt{k+1/2}\, |\langle f-f^\delta, \varphi_{k} \rangle |
\mathop{{\sum}}\limits_{l_1=r-1}^{k-1} (l_1+1/2) \mathop{{\sum}}\limits_{l_2=r-2}^{l_1-1} (l_2+1/2)
$$
$$
\ldots \mathop{{\sum}}\limits_{l_{r-1}=1}^{l_{r-2}-1} (l_{r-1}+1/2)\mathop{{\sum}}\limits_{l_r=0}^{l_{r-1}-1} (l_r+1/2)
$$
$$
\leq c\, \sum_{k=r}^N k^{2r+1/2}\, |\langle f-f^\delta, \varphi_{k} \rangle |
\leq c\, \left(\sum_{k=r}^{N}\,|\xi_k|^p\right)^{1/p}
\left(\sum_{k=r}^{N}\ k^{(2r+1/2)\frac{p}{p-1}}\right)^{(p-1)/p}
$$
$$
\leq c \delta N^{2r-1/p+3/2} ,
$$
which was required to prove.

In the cases of $p=1$ and $p=\infty$, the assertion of Lemma is proved similarly. \vspace{0.1in}
\end{proof}

The combination of Lemmas \ref{lemma_BoundErrHCC} and \ref{lemma_BoundPertHCC} gives

\begin{theorem} \label{Th2}
	Let $f\in W^\mu_{s}$, $1\leq s< \infty$, $\mu>2r-1/s+3/2$, and let the condition (\ref{perturbation2}) be satisfied for $1\le p\le \infty$. Then for $N\asymp \delta^{-\frac{1}{\mu-1/p+1/s}}$ it holds
	$$
	\|f^{(r)}-\mathcal{D}^{(r)}_N f^\delta\|_C \leq c \delta^{\frac{\mu-2r+1/s-3/2}{\mu-1/p+1/s}}.
	$$
\end{theorem}

\begin{proof}
Taking into account Lemmas \ref{lemma_BoundErrHCC}, \ref{lemma_BoundPertHCC}, from (\ref{fullError}) we get
$$
\|f^{(r)} - \mathcal{D}^{(r)}_N f^\delta\|_C \leq \|f^{(r)}-\mathcal{D}_N^{(r)} f\|_C +
\|\mathcal{D}^{(r)}_N f - \mathcal{D}^{(r)}_N f^\delta\|_C
$$
$$
\leq c\|f\|_{s,\mu} N^{-\mu+2r-1/s+3/2} + c\delta N^{2r-1/p+3/2}
= cN^{2r+3/2} \left(N^{-\mu-1/s} + \delta N^{-1/p}\right).
$$
Substituting the rule $N\asymp \delta^{-\frac{1}{\mu-1/p+1/s}}$ into the relation above completely proves Theorem.
\end{proof}

In what follows, by $\card([r,N])$ we shall understand the number of points with integer indexes that belong to the interval $[r,N]$.

\begin{corollary} \label{Cor2}
In the considered problem, the truncation method $\mathcal{D}^{(r)}_{N}$ (\ref{ModVer})
achieves the accuracy
	$$
	O\left(\delta^{\frac{\mu-2r+1/s-3/2}{\mu-1/p+1/s}}\right)
	$$
	on the class $W^{\mu}_{s}$, $\mu>2r-1/s+3/2$, and requires
	$$
	\card([r,N]) \asymp
	N \asymp \delta^{-\frac{1}{\mu-1/p+1/s}}
	$$
	perturbed Fourier-Legendre coefficients.
\end{corollary}

\section{Error estimate for truncation method in $L_2$--metric}

Now we have to bound the error of the method (\ref{ModVer}) in the metric of $L_2$. An upper estimate for the norm of the first difference on the right-hand side of (\ref{fullError}) is contained in the following statement.

\begin{lemma}\label{lemma_BoundErrHC}
	Let $f\in W^\mu_{s}$, $1\leq s< \infty$, $\mu>2r-1/s+1/2$. Then it holds
	$$
	\|f^{(r)}-\mathcal{D}_N^{(r)} f\|_2 \leq c\|f\|_{s,\mu} N^{-\mu+2r-1/s+1/2}.
	$$
\end{lemma}

\begin{proof}
We take the representation
$$
f^{(r)}(t)-\mathcal{D}^{(r)}_N f(t)
= 2^r \sum_{k=N+1}^\infty \sqrt{k+1/2}\,\langle f, \varphi_k\rangle
\mathop{{\sum}^*}\limits_{l_1=r-1}^{k-1} (l_1+1/2) \mathop{{\sum}^*}\limits_{l_2=r-2}^{l_1-1} (l_2+1/2) \cdots
\mathop{{\sum}^*}\limits_{l_r=0}^{l_{r-1}-1} \sqrt{l_r+1/2}\,\varphi_{l_r}(t)
$$
and change the order of summation:
\begin{align*}
&f^{(r)}(t)-\mathcal{D}^{(r)}_N f(t)=
2^r \sum_{k=N+1}^\infty \sqrt{k+1/2}\langle f, \varphi_k\rangle \mathop{{\sum}^*}\limits_{l_1=r-1}^{k-1} (l_1+1/2) \, \\[10pt]
&\ldots \mathop{{\sum}^*}_{l_{r-1}=1}^{l_{r-2}-1} (l_{r-1} + 1/2) \mathop{{\sum}^*}_{l_r=0}^{l_{r-1}-1} \sqrt{l_r + 1/2}\ \varphi_{l_r}(t) = 2^r \sum_{k=N+1}^\infty \sqrt{k+1/2}\,\langle f,\varphi_k\rangle \\[10pt]
&\ldots \mathop{{\sum}^*}_{l_{r-2}=2}^{l_{r-3}-1} (l_{r-2} + 1/2) \mathop{{\sum}^*}_{l_{r}=0}^{l_{r-2}-1} \sqrt{l_r + 1/2}\ \varphi_{l_r}(t) \mathop{{\sum}^*}_{l_{r-1}=l_r+1}^{l_{r-2}-1}(l_{r-1} + 1/2) \\[10pt]
&= \ldots = 2^r \sum_{k=N+1}^{\infty} \sqrt{k + 1/2}\,\langle f, \,\varphi_k\rangle \mathop{{\sum}^*}_{l_1=r-1}^{k-1} (l_1 + 1/2) \mathop{{\sum}^*}_{l_r=0}^{l_1-r+1} \sqrt{l_r + 1/2}\,\varphi_{l_r}(t)\mathop{{\sum}^*}_{l_2=l_r+r-2}^{l_1-1} (l_2 + 1/2)  \\
& \ldots \mathop{{\sum}^*}_{l_{r-2}=l_r+2}^{l_{r-3}-1} (l_{r-2} + 1/2) \mathop{{\sum}^*}_{l_{r-1}=l_r+1}^{l_{r-2}-1} (l_{r-1} + 1/2) \\[10pt]
&= 2^r \sum_{k=N+1}^{\infty} \sqrt{k + 1/2}\,\langle f, \varphi_k\rangle \mathop{{\sum}^*}_{l_r=0}^{k-r} \sqrt{l_r + 1/2} \ \varphi_{l_r}(t) \mathop{{\sum}^*}_{l_1=l_r+r-1}^{k-1}(l_1+1/2) \mathop{{\sum}^*}_{l_2=l_r+r-2}^{l_1-1} (l_2 + 1/2)  \\
& \ldots \mathop{{\sum}^*}_{l_{r-2}=l_r+2}^{l_{r-3}-1} (l_{r-2} + 1/2) \mathop{{\sum}^*}_{l_{r-1}=l_r+1}^{l_{r-2}-1} (l_{r-1} + 1/2).
\end{align*}
By interchanging the order of summation with respect to $k$ and $l_r$, we obtain
$$
f^{(r)}(t)-\mathcal{D}^{(r)}_N f(t) = 2^r \mathop{{\sum}^*}_{l_r=0}^{N-r+1} \sqrt{l_r+1/2} \ \varphi_{l_r}(t)
\sum_{k=N+1}^\infty \sqrt{k+1/2}\langle f, \varphi_k\rangle B_k^r
$$
$$
+ 2^r \mathop{{\sum}^*}_{l_r=N-r+2}^{\infty} \sqrt{l_r+1/2} \ \varphi_{l_r}(t)
\sum_{k=l_r+r}^\infty \sqrt{k+1/2}\langle f, \varphi_k\rangle B_k^r ,
$$
where $B^1_k=1$ and for r= 2, 3,\dots
$$
B^r_{k}:=\mathop{{\sum}^*}\limits_{l_1=l_r+r-1}^{k-1} (l_1+1/2) \mathop{{\sum}^*}\limits_{l_2=l_r+r-2}^{l_1-1}
(l_2+1/2) \ldots \mathop{{\sum}^*}\limits_{l_{r-2}=l_r+2}^{l_{r-3}-1}  (l_{r-2}+1/2)\,
\mathop{{\sum}^*}\limits_{l_{r-1}=l_r+1}^{l_{r-2}-1}  (l_{r-1}+1/2) .
$$
Easy to see that
\begin{equation}   \label{Br}
B^r_{k} = c^* k^{2r-2} ,
\end{equation}
where the factor $c^*$ does not depend on $N$, $k$.

At first, we consider the case $1<s<\infty$. Using Hölder inequality and the definition of $W^\mu_s$, for $\mu>2r-1/s+1/2$ we get
$$
\|f^{(r)}-\mathcal{D}^{(r)}_N f\|^2_{2}
= 2^{2r}\, \mathop{{\sum}^*}_{l_{r}=0}^{N-r+1} (l_r+1/2)
\left(\sum_{k=N+1}^ \infty \sqrt{k+1/2}\, \langle f, \varphi_{k} \rangle B^r_{k}\right)^2
$$
$$
+ 2^{2r}\, \mathop{{\sum}^*}_{l_{r}=N-r+2}^{\infty}(l_r+1/2)
\left(\sum_{k=l_r+r}^\infty \sqrt{k+1/2}\, \langle f, \varphi_{k} \rangle B^r_{k}\right)^2
$$
$$
\leq c \mathop{{\sum}}\limits_{l_{r}=0}^{N-r+1} (l_r+1/2)
\left(\sum_{k=N+1}^\infty k^{s\mu}\, |\langle f, \varphi_{k} \rangle |^s\right)^{2/s}
\left(\sum_{k=N+1}^\infty k^{(2r-\mu-3/2)\frac{s}{s-1}}\right)^{2(s-1)/s}
$$
$$
+ c\, \mathop{{\sum}}\limits_{l_{r}=N-r+2}^{\infty} (l_r+1/2)
\left(\sum_{k=l_r+r}^\infty k^{s\mu}\, |\langle f, T_{k} \rangle |^s\right)^{2/s}
\left(\sum_{k=l_r+r}^\infty \, k^{(2r-\mu-3/2)\frac{s}{s-1}}\right)^{2(s-1)/s}
$$
$$
\leq c \|f\|_{s,\mu}^2 N^{-2(\mu-2r+3/2)+\frac{2(s-1)}{s}+2}.
$$
Thus, we find
$$
\|f^{(r)}-\mathcal{D}^{(r)}_N f\|_2 \leq c\|f\|_{s,\mu} N^{-\mu+2r-1/s+1/2}.
$$
In the case of $s = 1$ the assertion of Lemma is proved similarly.
\end{proof}

\begin{lemma}\label{lemma_BoundPertHC}
Let the condition (\ref{perturbation2}) be satisfied for $1\le p\le \infty$.
Then for arbitrary function $f\in L_2$ it holds
	$$
	\|\mathcal{D}^{(r)}_N f - \mathcal{D}^{(r)}_N f^\delta\|_2 \leq c\delta N^{2r-1/p+1/2}.
	$$
\end{lemma}

\begin{proof}
We take the representation (\ref{second_dif}) and change the order of summation:

\begin{align*}
&\mathcal{D}^{(r)}_N f(t) - \mathcal{D}^{(r)}_N f^\delta(t)=
2^r \sum_{k=r}^N \sqrt{k+1/2}\langle f-f^\delta, \varphi_k\rangle
\mathop{{\sum}^*}_{l_1=r-1}^{k-1} (l_1+1/2) \mathop{{\sum}^*}_{l_2=r-2}^{l_1-1} (l_2+1/2) \cdots
\mathop{{\sum}^*}_{l_r=0}^{l_{r-1}-1} \sqrt{l_r+1/2}\varphi_{l_r}(t) \\[10pt]
&= 2^r \sum_{k=r}^N \sqrt{k+1/2}\langle f-f^\delta,\varphi_k\rangle \dots \mathop{{\sum}^*}_{l_{r-2}=2}^{l_{r-3}-1} (l_{r-2} + 1/2) \mathop{{\sum}^*}_{l_{r}=0}^{l_{r-2}-2} \sqrt{l_r + 1/2}\ \varphi_{l_r}(t) \mathop{{\sum}^*}_{l_{r-1}=l_r+1}^{l_{r-2}-1}(l_{r-1} + 1/2)  \\[10pt]
&= \dots = 2^r \sum_{k=r}^{N} \sqrt{k + 1/2}\langle f-f^\delta, \varphi_k\rangle \mathop{{\sum}^*}_{l_1=r-1}^{k-1} (l_1 + 1/2) \mathop{{\sum}^*}_{l_r=0}^{l_1-r+1} \sqrt{l_r + 1/2}\varphi_{l_r}(t)\mathop{{\sum}^*}_{l_2=l_r+r-2}^{l_1-1} (l_2 + 1/2)  \\
& \dots \mathop{{\sum}^*}_{l_{r-2}=l_r+2}^{l_{r-3}-1} (l_{r-2} + 1/2) \mathop{{\sum}^*}_{l_{r-1}=l_r+1}^{l_{r-2}-1} (l_{r-1} + 1/2) \\[10pt]
&= 2^r \sum_{k=r}^{N} \sqrt{k + 1/2}\langle f-f^\delta, \varphi_k\rangle \mathop{{\sum}^*}_{l_r=0}^{k-r} \sqrt{l_r + 1/2} \ \varphi_{l_r}(t) \mathop{{\sum}^*}_{l_1=l_r+r-1}^{k-1}(l_1+1/2) \mathop{{\sum}^*}_{l_2=l_r+r-2}^{l_1-1} (l_2 + 1/2)  \\
& \dots \mathop{{\sum}^*}_{l_{r-2}=l_r+2}^{l_{r-3}-1} (l_{r-2} + 1/2) \mathop{{\sum}^*}_{l_{r-1}=l_r+1}^{l_{r-2}-1} (l_{r-1} + 1/2).
\end{align*}

As a result we have
$$
\mathcal{D}^{(r)}_N f(t) - \mathcal{D}^{(r)}_N f^\delta(t)
= 2^r \mathop{{\sum}^*}_{l_r=0}^{N-r} \sqrt{l_r+1/2} \  \varphi _{l_r}(t)
\, \sum_{k=l_r+r}^N \sqrt{k+1/2} \, \langle f-f^\delta, \varphi_{k} \rangle  B^r_k .
$$

At first, we consider the case $1<p<\infty$. Using Hölder inequality, we get
$$
\|\mathcal{D}^{(r)}_N f - \mathcal{D}^{(r)}_N f^\delta\|_{2}^2
\le 2^{2r} \mathop{{\sum}}\limits_{l_r=0}^{N-r}
\, (l_r+1/2) \left(\sum_{k=l_r+r}^N \sqrt{k+1/2}\, \,|\langle f-f^\delta, \varphi_{k} \rangle| B^r_k\right)^2
$$
$$
\leq c\, \mathop{{\sum}}\limits_{l_r=0}^{N-r} (l_r+1/2)
\, \left(\sum_{k=l_r+r}^N |\xi_k|^p\right)^{2/p}
\, \left(\sum_{k=l_r+r}^N k^{(2r-3/2)\frac{p}{p-1}}\right)^{2(p-1)/p}
$$
$$
\leq c\, \delta^2 N^{2(2r-3/2)+2(p-1)/p+2}.
$$
In the case of $p = 1$ the assertion of Lemma is proved similarly.
\end{proof}

The combination of Lemmas \ref{lemma_BoundErrHC} and \ref{lemma_BoundPertHC} gives

\begin{theorem} \label{Th1}
Let $f\in W^\mu_{s}$, $1\leq s< \infty$, $\mu>2r-1/s+1/2$, and let the condition (\ref{perturbation2}) be satisfied for $1\le p\le \infty$. Then for $N\asymp \delta^{-\frac{1}{\mu-1/p+1/s}}$ it holds
	$$
	\|f^{(r)} - \mathcal{D}^{(r)}_N f^\delta\|_2 \leq c \delta^{\frac{\mu-2r+1/s-1/2}{\mu-1/p+1/s}}.
	$$
\end{theorem}

\begin{corollary} \label{Cor1}
In the considered problem, the truncation method $\mathcal{D}^{(r)}_{N}$ (\ref{ModVer})
achieves the accuracy
	$$
	O\left(\delta^{\frac{\mu-2r+1/s-1/2}{\mu-1/p+1/s}}\right)
	$$
	on the class $W^{\mu}_{s}$, $\mu>2r-1/s+1/2$, and requires
	$$
	\card([r,N]) \asymp N \asymp \delta^{-\frac{1}{\mu-1/p+1/s}}
	$$
perturbed Fourier-Legendre coefficients.
\end{corollary}

\section{Error estimate in the $L_{q}$--metrics}  \label{L_q}

In the preceding sections, we have established upper bounds for the accuracy of the truncation method $\mathcal{D}^{(r)}_{N}$ (\ref{ModVer}) when measured in the $C$ and $L_{2}$ metrics. Now, we extend our analysis to derive comparable bounds across the complete spectrum of spaces $L_{q}$ for $2\le q\le \infty$.

To facilitate this extension, we introduce several key facts and definitions. A well-established result
(see, for example, Theorem 6.6 \cite{DT}) states that for any $2 \le q \le \infty$ and any algebraic polynomial $P_N$ of degree $N$:
\begin{equation}  \label{P_N}
\|P_N\|_{q} \le c\, N^{(1-\frac{2}{q})}\, \|P_N\|_{2} .
\end{equation}
Let us define the best approximation quantity as
$$
\mathbb{E}_N (f)_{2} = \inf\limits_{P\in\Pi_N} \|f - P\|_{2} ,
$$
where $\Pi_N$ represents the collection of polynomials with degree $\le N$.

Theorem 11.1 in \cite{DT} establishes that for any function $g\in L_{2}$ and $2 < q \le \infty$, the following inequality holds:
\begin{equation}  \label{g}
\|g\|_{q} \le c\,
\left[ \left\{\sum_{k=1}^\infty  k^{(1-\frac{2}{q}) q_1-1}\,\mathbb{E}_k(g)_{2}^{q_1} \right\}^{1/q_1} +
\|g\|_{2} \right] ,
\end{equation}
where $q_1$ is defined as:
$$
q_1 = \begin{cases}
q, & \text{if } q<\infty , \\
1, & \text{if } q=\infty .
\end{cases}
$$

We now present an upper estimate for the first difference term from the right-hand side of equation (\ref{fullError}) in the $L_{q}$ metrics.

\begin{lemma}\label{lemma_Bound1_L_q}
Let $f\in W^\mu_{s}$, $1\leq s< \infty$, $\mu>2r-1/s-2/q+3/2$, $2\leq q\le \infty$.
Then the following bound holds:
$$
\|f^{(r)}-\mathcal{D}_N^{(r)} f\|_{q} \leq c\,
N^{-\mu+2r-1/s-2/q+3/2} .
$$
\end{lemma}

\begin{proof}
Let us define $g=f^{(r)}-\mathcal{D}_N^{(r)}f$. According to Lemma \ref{lemma_BoundErrHC}, we have:
\begin{equation}   \label{g_2}
\|g\|_{2} = \|f^{(r)}-\mathcal{D}_N^{(r)} f\|_{2} \leq
c\, \|f\|_{s,\mu} N^{-\mu+2r-1/s+1/2} .
\end{equation}

First, consider the case $2<q< \infty$. From (\ref{g}), we derive:
\begin{equation}  \label{g_q}
\|g\|_{q} \le c\,
\left[ \left\{\sum_{k=1}^\infty  k^{q-3}\,\mathbb{E}_k(g)_{2}^{q} \right\}^{1/q} +
\|g\|_{2}\right] .
\end{equation}

We separate this summation by $k$ into two components:
$$
S_1 = \sum_{k=1}^{N-r}  k^{q-3}\,\mathbb{E}_k(g)_{2}^{q} ,
$$
$$
S_2 = \sum_{k=N-r+1}^{\infty}  k^{q-3}\,\mathbb{E}_k(g)_{2}^{q} .
$$

For the polynomial $P$ in $\mathbb{E}_k(g)_{2}$, we select:
\begin{equation}  \label{P_k}
P = \begin{cases}
0, & \text{if } 1\le k \le N-r ,\\
\mathcal{D}_k^{(r)} f - \mathcal{D}_N^{(r)} f, & \text{if } k\ge N-r+1 .
\end{cases}
\end{equation}

Using relation (\ref{g_q}), we can bound $S_1$ as follows:
$$
S_1 \le \sum_{k=1}^{N-r}  k^{q-3}\, \|f^{(r)}-\mathcal{D}_N^{(r)} f\|_{2}^q
$$
$$
\le c\, N^{-(\mu-2r+1/s-1/2)q}\, \sum_{k=1}^{N-r}  k^{q-3}
\le c\, N^{-(\mu-2r+1/s+2/q-3/2)q} .
$$

For $\mu>2r-1/s-2/q+3/2$, we can bound $S_2$ as:
$$
S_2 \le \sum_{k=N-r+1}^{\infty}  k^{q-3}\, \|f^{(r)}-\mathcal{D}_k^{(r)} f\|_{2}^q
$$
$$
\le c\, \sum_{k=N-r+1}^{\infty}  k^{-(\mu-2r+1/s+2/q-3/2)q-1}
\le c\, N^{-(\mu-2r+1/s+2/q-3/2)q} .
$$

Consequently, we have:
\begin{equation}   \label{S_1+S_2}
(S_1+S_2)^{1/q} \le c\, N^{-\mu+2r-1/s-2/q+3/2}
\end{equation}

Substituting (\ref{g_2}) and (\ref{S_1+S_2}) into relation (\ref{g_q}), we obtain:
$$
\|f^{(r)}-\mathcal{D}_N^{(r)} f\|_{q} \leq
c\, (N^{-\mu+2r-1/s-2/q+3/2} + N^{-\mu+2r-1/s+1/2})
= c\, N^{-\mu+2r-1/s-2/q+3/2} .
$$

Now consider the case where $q=\infty$. From inequality (\ref{g}), we have:
\begin{equation}  \label{g_infty}
\|g\|_{\infty} \le c\,
\Big(\bar{S}_1 + \bar{S}_2 + \|g\|_{2}\Big) ,
\end{equation}
where
$$
\bar{S}_1 = \sum_{k=1}^{N-r} \mathbb{E}_k(g)_{2} ,
$$
$$
\bar{S}_2 = \sum_{k=N-r+1}^{\infty} \mathbb{E}_k(g)_{2} .
$$

As before, we use $P$ from (\ref{P_k}) in $\mathbb{E}_k(g)_{2}$.
Using estimation (\ref{g_2}), we derive:
$$
\bar{S}_1 \le \sum_{k=1}^{N-r} \|f^{(r)}-\mathcal{D}_N^{(r)} f\|_{2}
$$
$$
\le c\, N^{-\mu+2r-1/s+1/2}\, \sum_{k=1}^{N-r}  1
= c\, N^{-\mu+2r-1/s+3/2} .
$$
For $\mu>2r-1/s+3/2$, we have
$$
\bar{S}_2 \le \sum_{k=N-r+1}^{\infty}  \|f^{(r)}-\mathcal{D}_k^{(r)} f\|_{2}
$$
$$
\le c\, \sum_{k=N-r+1}^{\infty}  k^{-\mu+2r-1/s+1/2}
= c\, N^{-\mu+2r-1/s+3/2} .
$$
This yields
\begin{equation}   \label{bar(S_1+S_2)}
\bar{S}_1+\bar{S}_2 \le c\, N^{-\mu+2r-1/s+3/2} .
\end{equation}
Substituting (\ref{g_2}) and (\ref{bar(S_1+S_2)}) into inequality (\ref{g_infty}), we obtain
$$
\|f^{(r)}-\mathcal{D}_N^{(r)} f\|_{\infty} \leq
c\, N^{-\mu+2r-1/s+3/2} .
$$
For the case $q=2$, the result was established in Lemma \ref{lemma_BoundErrHC}. \\
Thus, Lemma is fully proven for all applicable values of $q$.
\end{proof}

Next, we provide an estimate for the second difference term from the right-hand side of (\ref{fullError})
in the $L_{q}$-metric.

\begin{lemma}\label{lemma_Bound2_L_q}
Let the condition (\ref{perturbation2}) be satisfied for $1\le p\le \infty$.
Then for any function $f\in L_{2}$ and $2\le q\le\infty$, the following bound holds:
$$
\|\mathcal{D}^{(r)}_N f - \mathcal{D}^{(r)}_N f^\delta\|_{q}
\leq c\,\delta\, N^{2r-1/p-2/q+3/2} .
$$
\end{lemma}

\begin{proof}
From Lemma \ref{lemma_BoundPertHC}, we know:
$$
\|\mathcal{D}^{(r)}_N f - \mathcal{D}^{(r)}_N f^\delta\|_{2}
\leq c\,\delta\, N^{2r-1/p+1/2} .
$$
Applying inequality (\ref{P_N}) for any $2\le q\le \infty$, we obtain
$$
\|\mathcal{D}^{(r)}_N f - \mathcal{D}^{(r)}_N f^\delta\|_{q}
\leq c\, \delta\, N^{2r-1/p-2/q+3/2} ,
$$
which completes the proof.
\end{proof}

By combining Lemmas \ref{lemma_Bound1_L_q} and \ref{lemma_Bound2_L_q}, we arrive at the following statement.

\begin{theorem} \label{Th1_L_q}
Let $f\in W^\mu_{s}$, $1\leq s< \infty$, $\mu>2r-1/s-2/q+3/2$, $2\leq q\le \infty$, and
let the condition (\ref{perturbation2}) be satisfied for $1\le p\le \infty$.
Then for $N\asymp \delta^{-\frac{1}{\mu-1/p+1/s}}$ it holds
$$
\|f^{(r)} - \mathcal{D}^{(r)}_N f^\delta\|_{q}
\leq c\, \delta^{\frac{\mu-2r+1/s+2/q-3/2}{\mu-1/p+1/s}} .
$$
\end{theorem}

\begin{proof}
Applying Lemmas \ref{lemma_Bound1_L_q} and \ref{lemma_Bound2_L_q} to equation (\ref{fullError}), we obtain
$$
\|f^{(r)} - \mathcal{D}^{(r)}_N f^\delta\|_{q}
\leq \|f^{(r)}-\mathcal{D}_N^{(r)} f\|_{q} +
\|\mathcal{D}^{(r)}_N f - \mathcal{D}^{(r)}_N f^\delta\|_{q}
$$
$$
\leq c\, N^{-\mu+2r-1/s-2/q+3/2} + c\, \delta N^{2r-1/p-2/q+3/2}
$$
$$
= c\, N^{2r-2/q+3/2} \left(N^{-\mu-1/s} + \delta N^{-1/p}\right) .
$$
By substituting $N\asymp \delta^{-\frac{1}{\mu-1/p+1/s}}$ into this expression, we complete the proof of Theorem.
\end{proof}

\begin{corollary} \label{Cor3}
For the problem under consideration, the truncation method $\mathcal{D}^{(r)}_{N}$ (\ref{ModVer}) achieves an accuracy of:
$$
O\Big(\delta^{\frac{\mu-2r+1/s+2/q-3/2}{\mu-1/p+1/s}}\Big)
$$
on the class $W^{\mu}_{s}$, $\mu>2r-1/s-2/q+3/2$, requiring
$$
\card([r,N]) \asymp
N \asymp \delta^{-\frac{1}{\mu-1/p+1/s}}
$$
perturbed Fourier-Legendre coefficients.
\end{corollary}

\begin{remark} \label{comp1}
In previous research (see \cite{Lu&Naum&Per}), error estimates for the truncation method
$\mathcal{D}^{(r)}_{N}$ (\ref{ModVer}) were only established for the specific case where $r=1$, $p=s=2$,
with output spaces limited to $C$ and $L_{2}$.
Thus, Theorems \ref{Th2} and \ref{Th1_L_q} significantly expand upon these earlier findings by generalizing
to arbitrary values of $r$, $p$, $s$, and $q$.
\end{remark}

\end{document}